\documentclass[preprint,12pt]{elsarticle}

\usepackage{amssymb}
\usepackage{amsthm}

\newcommand{\sysn}{\left\{\begin{array}{rcl}}
\newcommand{\sysk}{\end{array}\right.}

\renewcommand{\ge}{\geqslant}

\newtheorem{theorem}{Theorem}[section]

\theoremstyle{example}

\theoremstyle{definition}

\journal{...}

\begin{document}

\begin{frontmatter}

%% Title, authors and addresses

%% use the tnoteref command within \title for footnotes;
%% use the tnotetext command for the associated footnote;
%% use the fnref command within \author or \address for footnotes;
%% use the fntext command for the associated footnote;
%% use the corref command within \author for corresponding author footnotes;
%% use the cortext command for the associated footnote;
%% use the ead command for the email address,
%% and the form \ead[url] for the home page:
%%
%%\title{Topological-Algebraic Properties of Function Space with Set-Open Topology\tnoteref{label1}}
%%\tnotetext[label1]{}
%%\author{Alexander V. Osipov\corref{cor1}\fnref{label2}}
%%\ead{OAB@list.ru}
%% \ead[url]{home page}
%% \fntext[label2]{}
%% \cortext[cor1]{}
%% \address{Address\fnref{label3}}
%% \fntext[label3]{}

\title{Selection principles and games in bitopological function spaces}%\tnoteref{label1}}

%% use optional labels to link authors explicitly to addresses:
%% \author[label1,label2]{<author name>}
%% \address[label1]{<address>}
%% \address[label2]{<address>}

\author[label1]{Daniil Lyakhovets}

\ead[label1]{zoy01111@gmail.com}

%\tnotetext[label1]{The research has been supported by .}

\address[label1]{Krasovskii Institute of Mathematics and Mechanics, 620219, Yekaterinburg, Russia}

\author[label2]{Alexander~V.~Osipov}

\ead[label2]{OAB@list.ru}

\address[label2]{Krasovskii Institute of Mathematics and Mechanics, Ural Federal
 University, Ural State University of Economics, 620219, Yekaterinburg, Russia}

\begin{abstract}
%% Text of abstract
For a Tychonoff space $X$, we denote by $(C(X), \tau_k, \tau_p)$
the bitopological space of all real-valued continuous functions on
$X$ where $\tau_k$ is the compact-open topology and $\tau_p$ is
the topology of pointwise convergence. In papers
\cite{kooz,kooz1,os5} variations of selective separability and
tightness in $(C(X), \tau_k, \tau_p)$ were investigated. In this
paper we continued to study the selective properties and the
corresponding topological games in the space $(C(X), \tau_k,
\tau_p)$.

\end{abstract}

\begin{keyword}
%% keywords here, in the form: keyword \sep keyword
selection principles \sep compact-open topology \sep function
space \sep bitopological space \sep topological games \sep
separable space

%% MSC codes here, in the form: \MSC code \sep code

\MSC[2010]   54C35  \sep 54D65 \sep 54E55  \sep 54A20 \sep 91A05
\sep 91A44
%% or \MSC[2008] code \sep code (2000 is the default)

\end{keyword}

\end{frontmatter}

%%
%% Start line numbering here if you want
%%
% \linenumbers

%% main text

\section{Introduction}
\label{}

In papers \cite{bbm1,cmkm,mkm,os2,os4,os41,os3,sch} the authors
investigated the selectors of dense subsets of the space $C(X)$ of
all real-valued continuous functions on a Tychonoff space $X$ with
the topology $\tau_p$ of pointwise convergence and with the
compact-open topology $\tau_k$. For a Tychonoff space $X$, we
denote by $(C(X), \tau_k, \tau_p)$ the bitopological space. In
articles \cite{kooz,kooz1,os5}, variations of selective
separability and tightness in $(C(X), \tau_k, \tau_p)$ were
investigated. In this paper, we continued to study the selective
properties and the corresponding topological games in the space
$(C(X), \tau_k, \tau_p)$.  The following selection properties for
$(C(X), \tau_k, \tau_p)$ are considered.

\medskip
\begin{center}

$S_1(\mathcal{D}^k,\mathcal{S}^p) =
S_{fin}(\mathcal{D}^k,\mathcal{S}^p) \Rightarrow
S_1(\mathcal{D}^k,\mathcal{D}^p) \Rightarrow
S_{fin}(\mathcal{D}^k,\mathcal{D}^p)$

\end{center}

\bigskip

For example, a space $(C(X), \tau_k, \tau_p)$ satisfies
$S_1(\mathcal{D}^k,\mathcal{S}^p)$ (resp.,
$S_{fin}(\mathcal{D}^k,\mathcal{S}^p))$ if whenever $(D_n : n\in
\mathbb{N})$ is a sequence of dense subsets of $C_k(X)$, one can
take points $f_n\in D_n$ (resp., finite $F_n\subset D_n$) such
that $\{f_n : n\in \mathbb{N}\}$ (resp., $\bigcup \{F_n: n\in
\mathbb{N}\}$) is sequentially dense in $C_p(X)$. There is a
topological game, denoted by $G_{*}(\mathcal{A},\mathcal{B})$,
 corresponding to $S_{*}(\mathcal{A},\mathcal{B})$.

In this paper, we have gave characterizations for the
bitopological space $(C(X), \tau_k, \tau_p)$ to satisfy the
selection properties and the corresponding games.

\section{Main definitions and notation}

 Let $\mathcal{A}$ and $\mathcal{B}$ be sets consisting of
families of subsets of an infinite set $X$. Then many topological
properties are characterized in terms  of the following classical
selection principles:

$S_{1}(\mathcal{A},\mathcal{B})$ is the selection hypothesis: for
each sequence $(A_{n}: n\in \mathbb{N})$ of elements of
$\mathcal{A}$ there is a sequence $(b_{n} : n\in \mathbb{N})$ such
that for each $n$, $b_{n}\in A_{n}$, and $\{b_{n}: n\in\mathbb{N}
\}$ is an element of $\mathcal{B}$.

$S_{fin}(\mathcal{A},\mathcal{B})$ is the selection hypothesis:
for each sequence $(A_{n}: n\in \mathbb{N})$ of elements of
$\mathcal{A}$ there is a sequence $(B_{n}: n\in \mathbb{N})$ of
finite sets such that for each $n$, $B_{n}\subseteq A_{n}$, and
$\bigcup_{n\in\mathbb{N}}B_{n}\in\mathcal{B}$.

\medskip

The following prototype of many classical properties is called
"$\mathcal{A}$ choose $\mathcal{B}$" in \cite{tss}.

${\mathcal{A}\choose\mathcal{B}}$ : For each $\mathcal{U}\in
\mathcal{A}$ there exists $\mathcal{V}\subseteq \mathcal{U}$ such
that $\mathcal{V}\in \mathcal{B}$. In this paper we accept that
$|\mathcal{V}|=\aleph_0$.

Then $S_{fin}(\mathcal{A},\mathcal{B})$ implies
${\mathcal{A}\choose\mathcal{B}}$.

In this paper, by a cover we mean a nontrivial one, that is,
$\mathcal{U}$ is a cover of $X$ if $X=\bigcup \mathcal{U}$ and
$X\notin \mathcal{U}$.

An open cover $\mathcal{U}$ of a space $X$ is called:

$\bullet$  an $\omega$-cover (a $k$-cover) if each finite
(compact) subset $C$ of $X$ is contained in an element of
$\mathcal{U}$.

$\bullet$  a $\gamma$-cover (a $\gamma_k$-cover) if $\mathcal{U}$
is infinite and for each finite (compact) subset $C$ of $X$ the
set $\{U\in \mathcal{U} : C\nsubseteq U\}$ is finite.

For a topological space $X$ we denote:

$\bullet$ $\mathcal{O}$ --- the family of all open covers of $X$;

$\bullet$ $\Gamma$ --- the family of all open $\gamma$-covers of
$X$;

$\bullet$ $\Gamma_k$ --- the family of all open $\gamma_k$-covers
of $X$;

$\bullet$ $\Omega$ --- the family of all open $\omega$-covers of
$X$;

$\bullet$ $\mathcal{K}$ --- the family of all open $k$-covers of
$X$;

$\bullet$ $\mathcal{D}^k$ --- the family of all dense subsets of
$C_k(X)$;

$\bullet$ $\mathcal{D}^p$ --- the family of all dense subsets of
$C_p(X)$;

$\bullet$ $\mathcal{S}^k$ --- the family of all sequentially dense
subsets of $C_k(X)$;

$\bullet$ $\mathcal{S}^p$ --- the family of all sequentially dense
subsets of $C_p(X)$;

$\bullet$ $\mathbb{K}(X)$ --- the family of all non-empty compact
subsets of $X$;

$\bullet$ $\mathbb{F}(X)$ --- the family of all non-empty finite
subsets of $X$.

\bigskip
A space $X$ is said to be a $\gamma_k$-set if each $k$-cover
$\mathcal{U}$ of $X$ contains a countable set $\{U_n : n\in
\mathbb{N}\}$ which is a $\gamma_k$-cover of $X$ \cite{koc1}.

 If $X$ is a space and $A\subseteq X$, then the sequential closure of $A$,
 denoted by $[A]_{seq}$, is the set of all limits of sequences
 from $A$. A set $D\subseteq X$ is said to be sequentially dense
 if $X=[D]_{seq}$. A space $X$ is called sequentially separable if
 it has a countable sequentially dense set. Clearly, every sequentially separable space is
 separable.

\medskip

%$\bullet$ A space $X$ is $R$-separable, if $X$ satisfies
%$S_1(\mathcal{D}, \mathcal{D})$ (Def. 47, \cite{bbm1}).

%$\bullet$ A space $X$ is $M$-separable (selective separability),
%if $X$ satisfies $S_{fin}(\mathcal{D}, \mathcal{D})$.

\medskip

Let $X$ be a topological space, and $x\in X$. A subset $A$ of $X$
{\it converges} to $x$, $x=\lim A$, if $A$ is infinite, $x\notin
A$, and for each neighborhood $U$ of $x$, $A\setminus U$ is
finite. Consider the following collection:

$\bullet$ $\Omega_x=\{A\subseteq X : x\in \overline{A}\setminus
A\}$;

$\bullet$ $\Gamma_x=\{A\subseteq X : x=\lim A\}$.

Note that if $A\in \Gamma_x$, then there exists $\{a_n\}\subset A$
converging to $x$. So, simply $\Gamma_x$ may be the set of
non-trivial convergent sequences to $x$.

\bigskip

We write $\Pi (\mathcal{A}_x, \mathcal{B}_x)$ without specifying
$x$, we mean $(\forall x) \Pi (\mathcal{A}_x, \mathcal{B}_x)$.

 So we have three types of topological properties of $(C(X), \tau_k, \tau_p)$
described through the selection principles of $X$ where the index
$k$ means the compact-open topology and the index $p$ - the
topology of pointwise convergence:

$\bullet$  local properties of the form $S_*(\Phi^k_x,\Psi^p_x)$;

$\bullet$  global properties of the form $S_*(\Phi^k,\Psi^p)$;

$\bullet$  semi-local properties of the form
$S_*(\Phi^k,\Psi^p_x)$.

\medskip

 There is a game, denoted by $G_{fin}(\mathcal{A},\mathcal{B})$,
 corresponding to $S_{fin}(\mathcal{A},\mathcal{B})$; two players,
 ONE and TWO, play a round for each natural number $n$. In the
 $n$-th round ONE chooses a set $A_n\in \mathcal{A}$ and TWO
 responds with a finite subset $B_n$ of $A_n$. A play
 $A_1,B_1;...;A_n,B_n;...$ is won by TWO if $\bigcup\limits_{n\in
 \mathbb{N}} B_n\in \mathcal{B}$; otherwise, ONE wins.

A strategy of a player is a function $\sigma$ from the set of all
finite sequences of moves of the opponent into the set of (legal)
moves of the strategy owner.

If ONE does not have a winning strategy in the game
$G_{*}(\mathcal{A},\mathcal{B})$, then the selection hypothesis
$S_{*}(\mathcal{A},\mathcal{B})$ is true; it is easy to prove. The
converse implication is not always true.

Similarly, one defines the game $G_{1}(\mathcal{A},\mathcal{B})$,
associated with $S_{1}(\mathcal{A},\mathcal{B})$.

So we have three types of topological games on $(C(X), \tau_k,
\tau_p)$ described through the selection principles (or
topological games) of $X$:

$\bullet$  local games of the form $G_*(\Phi^k_x,\Psi^p_x)$;

$\bullet$  global games of the form $G_*(\Phi^k,\Psi^p)$;

$\bullet$  semi-local games of the form $G_*(\Phi^k,\Psi^p_x)$.

The symbol ${\bf 0}$ denotes the constantly zero function in
$C(X)$.

\section{$S_{1}(\mathcal{D}^k,\mathcal{D}^p)$ and $G_{1}(\mathcal{D}^k,\mathcal{D}^p)$}

\begin{theorem}(Theorem 3.7 in \cite{os5} for $\lambda=k$ and $\mu=p$) For a space $X$ the following are
equivalent:

\begin{enumerate}

\item  $(C(X), \tau_k, \tau_p)$ has the property
$S_{1}(\Omega^{k}_{\bf 0},\Omega^{p}_{\bf 0})$;

\item  $X$ has the property $S_{1}(\mathcal{K},\Omega)$.

\end{enumerate}

\end{theorem}

\medskip
Recall that the $i$-weight $iw(X)$ of a space $X$ is the smallest
infinite cardinal number $\tau$ such that $X$ can be mapped by a
one-to-one continuous mapping onto a Tychonoff space of the weight
not greater than $\tau$.

\medskip

\begin{theorem} (Noble \cite{nob}) \label{th31} A space $C_{k}(X)$ is separable iff
$iw(X)=\aleph_0$.
\end{theorem}

Recall that a subset $A$ of a bitopological space
$(X,\tau_1,\tau_2)$ is bidense (double dense or short $d$-dense)
in $X$ if $A$ is dense in both $(X,\tau_1)$ and $(X,\tau_2)$
(\cite{br}). $(X,\tau_1,\tau_2)$ is {\it d-separable} if there is
a countable set $A$ which is $d$-dense in $X$. Note that if
$iw(X)=\aleph_0$, then $(C(X), \tau_k, \tau_p)$ is $d$-separable.

\medskip

\begin{theorem}\label{th1} For a space $X$ with $iw(X)=\aleph_0$ the following are equivalent:

\begin{enumerate}

\item  $(C(X), \tau_k, \tau_p)$ has the property
$S_{1}(\mathcal{D}^k,\mathcal{D}^p)$;

\item  $(C(X), \tau_k, \tau_p)$ has the property
$S_{1}(\mathcal{D}^k,\Omega^p_0)$;

\item  $(C(X), \tau_k, \tau_p)$ has the property
$S_{1}(\Omega^k_0,\Omega^p_0)$;

\item  $X$ has the property $S_{1}(\mathcal{K},\Omega)$;

\item  $(C(X), \tau_k, \tau_p)$ has the property
${\mathcal{D}^k\choose\mathcal{D}^p}$;

\item  ONE has no winning strategy in the game
$G_1(\mathcal{K},\Omega)$;

\item  ONE has no winning strategy in the game
$G_1(\mathcal{D}^k,\mathcal{D}^p)$;

\item  ONE has no winning strategy in the game $G_1(\Omega^k_{\bf
0},\Omega^p_0)$;

\item  ONE has no winning strategy in the game
$G_1(\mathcal{D}^k,\Omega^p_0)$.

\end{enumerate}

\end{theorem}

\begin{proof}

$(1)\Rightarrow(4)$. Let $(U_i^k: i\in \mathbb{N})$ be a sequence
$k$-covers of $X$ and let $D = \{f_s : s\in \mathbb{N} \}$ be a
countable dense set in $C_k(X)$. Consider $P_i :=
\{h^i_{L,W,f_s}\in C(X): h^i_{L,W,f_s} \upharpoonright L = f_s
\upharpoonright L, L \in \mathbb{K}(X), L \subset W :
 W\in U_i^k, h \upharpoonright (X \setminus W) = 1, f_s \in D\}$.
Note $P_i$ is a dense subset of $C_k(X)$ for each $i\in
\mathbb{N}$. Indeed fix $f\in C(X), K\in \mathbb{K}(X), \epsilon
> 0$, then $\langle f, K, \epsilon \rangle$ is neighborhood of
$f$. There is $W_k\in U_i^k$ such that $K\subset W_k$. Then there
exists $f_s\in D$ such that $f_s\in \langle f, K, \epsilon
\rangle$. Take $h^i_{K,W_k,f_s}\in \langle f, K, \epsilon
\rangle$.

Since $\{P_i: i\in \mathbb{N} \}$ is a countable set of dense sets
in $C_k(X)$, by $(1)$, there exists $\{p_i : i\in \mathbb{N}\}$
such that $p_i\in P_i$ and $\{p_i : i\in \mathbb{N}\}$ is a dense
set in $C_p(X)$. For $\{p_i = h_{L_i,W_i,f_s^i} : i\in
\mathbb{N}\}$, we have that $\{W_i : i\in \mathbb{N}\}$ is
$\omega$-cover of $X$. Indeed, let $M = \{x_1, x_2, ..., x_k\}$ is
a finite set in $X$. Consider $U = \langle {\bf 0}, M,
(-\frac{1}{2};\frac{1}{2}) \rangle$, then there exists $i'$ such
that $p_{i'}\in U$. It follows that $M\subset W_{i'}$.

$(3)\Rightarrow(2)$ is immediate.

$(4)\Rightarrow(3)$. Let $\{P_i: i\in \mathbb{N}\}$ such that
$P_i\in \Omega_0^k$. Fix $m\in \mathbb{N}$. Take $M_i^m :=
\{W_{i,m,h}=h^{-1}(-\frac{1}{m};\frac{1}{m}) : h\in P_i\}$, where
$W_{i,m,n}$ is nonempty. Note that $M_i^m$ is $k$ - cover of $X$.
Indeed for each $K\in \mathbb{K}, \langle {\bf 0}, K,
(-\frac{1}{m};\frac{1}{m}) \rangle$ there exists $p\in P_i : p\in
\langle {\bf 0}, K, (-\frac{1}{m};\frac{1}{m}) \rangle$, then $K
\subset W_{i,m,p_i} = p^{-1}(-\frac{1}{m};\frac{1}{m})$. Mean for
every $i,m$ $M_i^m$ is family $\mathcal{K}$ - covers in $X$. Then
by $(4)$, there is $\{W_{i,m,h_{i,m}} : i, m\in \mathbb{N}\}$ such
that $W_{i,m,h_{i,m}}\in M_i^m$ and $\{W_{i,m,h_{i,m}} : i, m\in
\mathbb{N}\}$ is $\omega$ - cover of $X$. Show that $\{h_{i,m} :
i, m\in \mathbb{N}\} \in \Omega_0^p$. Take an arbitrary $S =
\{x_1, x_2, ..., x_k\}$ and $\epsilon > 0$. Consider $\langle
\overline{0}, S, \epsilon \rangle$ - neighborhood of
$\overline{0}$. There is $m' : \frac{1}{m'} < \epsilon$. Since a
$\omega$ - cover of $X$ is large cover of $X$, then there are $m',
i'\in \mathbb{N}$ such that $S\subset W_{i',m',h_{i',m'}}$ and
$\frac{1}{m'} < \epsilon$, therefor $h_{i',m'}\in \langle {\bf 0},
S, \epsilon \rangle$.

$(2)\Rightarrow(1)$. Let $\{D_{i,j} : i, j \in \mathbb{N}\}$ be a
countable set of dense sets in $C_k(X)$. Let $D=\{d_i : i\in
\mathbb{N}\}$ is a countable dense set  in $C_k(X)$ for every
$i\in \mathbb{N}$. By $S_1(D^k,\Omega^p_{d_i})$ there exists
$\{d_{i,j} : j\in \mathbb{N}\}$ such that $d_{i,j}\in D_{i,j}$ and
$\{d_{i,j} : j\in \mathbb{N}\} \in \Omega^p_{d_i}$. Consider $M =
\{d_{i,j} : i,j, \in \mathbb{N} \}$. Prove that $M$ is dense in
$C_p(X)$. Fix $f\in C(X)$. Let $L = \{x_1,x_2,...,x_n\}$ be a
finite set of $X$ and $\epsilon > 0$. The set $\langle f, L,
\epsilon \rangle$ is a neighborhood of $f$, then there is
$d_{i'}\in D$ such that $d_{i'}\in \langle f, L, \epsilon
\rangle$, than there is
 $j'$ such that $d_{i',j'}\in \langle f, L, \epsilon \rangle$, than $M\in D^p$

 $(6)\Rightarrow(4)$ is immediate.

 $(4)\Rightarrow(6)$. Let $\sigma$ be a strategy for ONE in $G_1(\mathcal{K}, \Omega)$ and let the first move of ONE be a $k$-cover
 $\sigma(\emptyset)=\{U_{(\alpha^1)}: \alpha^1\in \Lambda^1\}$. Suppose that for each finite sequence $s$ of numbers $\alpha^i\in \Lambda^i$ of length
 at most $m$, $U_s$ has been already defined. Then define $\{U_{(\alpha^1,...,\alpha^m,\alpha^k)}: \alpha^k \in \Lambda^k\}$ to be the set
 $\sigma(U_{(\alpha^1)}, U_{(\alpha^1, \alpha^2)},...,U_{(\alpha^1,..., \alpha^m)})\setminus{\{U_{(\alpha^1)},U_{(\alpha^1, \alpha^2)},...,U_{(\alpha^1,...,\alpha^m)}\}}$.
 Because each compact subset of $X$ belongs to infinitely many elements of a $k$-cover, we have that,
 for each $s$, a finite sequence of numbers $\alpha^i\in \Lambda^i$, the set $\{U_{s\frown(\alpha^n)}: \alpha^n\in \Lambda^n\}$ is a $k$-cover.
 Apply (4) and, for each $s$, choose $\alpha^s\in \Lambda^s$ such that $\{U_{s\frown(\alpha^s)}: s$ a finite sequence of numbers $\alpha^i\in \Lambda^i, i\in\mathbb{N}$ $\}$
 is a $\omega$-cover of $X$. Then inductively define a sequence $\alpha^1=\alpha^{\emptyset}, \alpha^{k+1}=\alpha^{(\alpha^1,...,\alpha^k)}$ for $k \ge 1$.
 Then $U_{\alpha^1}, U_{\alpha^1,\alpha^2},...,U_{\alpha^1,...,\alpha^k},...$ is a $\omega$-cover, and because it is, in fact, a
 sequence of moves TWO in a play of game $G_1(\mathcal{K},\Omega)$, $\sigma$ is not a winning strategy for ONE.

Similarly to $(4)\Leftrightarrow(6)$ we have that
 $(1)\Leftrightarrow(7)$,  $(2)\Leftrightarrow(9)$ and
 $(3)\Leftrightarrow(8)$.

\end{proof}

\section{$S_{fin}(\mathcal{D}^k,\mathcal{D}^p)$ and $G_{fin}(\mathcal{D}^k,\mathcal{D}^p)$}

\begin{theorem}(Theorem 3.9 in \cite{os5} for $\lambda=k$ and $\mu=p$) For a space $X$ the following are
equivalent:

\begin{enumerate}

\item  $(C(X), \tau_k, \tau_p)$ has the property
$S_{fin}(\Omega^{k}_{\bf 0},\Omega^{p}_{\bf 0})$;

\item  $X$ has the property $S_{fin}(\mathcal{K},\Omega)$.

\end{enumerate}

\end{theorem}

\begin{theorem}\label{th111} For a space $X$ with $iw(X)=\aleph_0$ the following statements are equivalent:

\begin{enumerate}

\item  $(C(X), \tau_k, \tau_p)$ has the property
$S_{fin}(\mathcal{D}^k,\mathcal{D}^p)$;

\item  $(C(X), \tau_k, \tau_p)$ has the property
$S_{fin}(\mathcal{D}^k,\Omega^p_0)$;

\item  $(C(X), \tau_k, \tau_p)$ has the property
$S_{fin}(\Omega^k_0,\Omega^p_0)$;

\item  $X$ satisfies the selection principle $S_{fin}(\mathcal{K},\Omega)$;

\item  ONE has no winning strategy in the game
$G_{fin}(\mathcal{K},\Omega)$;

\item  ONE has no winning strategy in the game
$G_{fin}(\mathcal{D}^k,\mathcal{D}^p)$;

\item  ONE has no winning strategy in the game
$G_{fin}(\Omega^k_{\bf 0},\Omega^p_0)$;

\item  ONE has no winning strategy in the game
$G_{fin}(\mathcal{D}^k,\Omega^p_0)$.

\end{enumerate}

\end{theorem}

\begin{proof}

The implications are proved similarly to the proof of Theorem
\ref{th1}.

\end{proof}

\section{$S_{1}(\mathcal{D}^k,\mathcal{S}^p)$ and $G_{1}(\mathcal{D}^k,\mathcal{S}^p)$}

\begin{theorem}(Theorem 15 in \cite{cmkm})\label{th11} For a space $X$ the following are
equivalent:

\begin{enumerate}

\item  $(C(X), \tau_k, \tau_p)$ has the property
$S_{1}(\Omega^{k}_{\bf 0},\Gamma^{p}_{\bf 0})$;

\item  $X$ has the property $S_{1}(\mathcal{K},\Gamma)$.

\end{enumerate}

\end{theorem}

\begin{theorem}(Theorem 10 in \cite{mkm})\label{th14} For a space $X$ the following are
equivalent:

\begin{enumerate}

\item  $X$ has the property $S_{fin}(\mathcal{K},\Gamma)$;

\item  $X$ has the property $S_{1}(\mathcal{K},\Gamma)$;

\item  ONE has no winning strategy in the game
$G_1(\mathcal{K},\Gamma)$.

\end{enumerate}

\end{theorem}

\begin{theorem}\label{th1} For a space $X$ with $iw(X)=\aleph_0$ the following statements are equivalent:

\begin{enumerate}

\item  $(C(X), \tau_k, \tau_p)$ has the property
$S_{1}(\mathcal{D}^k,\mathcal{S}^p)$;

\item  $(C(X), \tau_k, \tau_p)$ has the property
${\mathcal{D}^k\choose\mathcal{S}^p}$;

\item  $X$ has the property $S_{1}(\mathcal{K},\Gamma)$;

\item $(C(X), \tau_k, \tau_p)$  has the property
$S_{fin}(\mathcal{D}^k,\mathcal{S}^p)$;

\item  $X$ has the property $S_{fin}(\mathcal{K},\Gamma)$;

\item Each finite power of $X$ has the property
$S_{1}(\mathcal{K},\Gamma)$;

\item $(C(X), \tau_k, \tau_p)$ has the property
$S_{1}(\Omega^k_{\bf 0},\Gamma^p_{\bf 0})$;

\item $(C(X), \tau_k, \tau_p)$ has the property
$S_{1}(\mathcal{D}^k,\Gamma^p_{\bf 0})$;

\item  $X$ has the property ${\mathcal{K}\choose\Gamma}$;

\item  ONE has no winning strategy in the game
$G_1(\mathcal{K},\Gamma)$;

\item  ONE has no winning strategy in the game
$G_1(\mathcal{D}^k,\mathcal{S}^p)$;

\item  ONE has no winning strategy in the game $G_1(\Omega^k_{\bf
0},\Gamma^p_{\bf 0})$;

\item  ONE has no winning strategy in the game
$G_1(\mathcal{D}^k,\Gamma^p_{\bf 0})$.

\end{enumerate}

\end{theorem}

\begin{proof}

By Theorem \ref{th11} (Theorem 15 in \cite{cmkm}),
$(3)\Leftrightarrow(7)$.

By Theorem \ref{th14} (Theorem 10 in \cite{mkm}),
$(3)\Leftrightarrow(5)\Leftrightarrow(10)$.

By Theorem 14 in \cite{cmkm}, $(3)\Leftrightarrow(9)$.

$(3)\Leftrightarrow(6)$ (Proposition 13 and Theorem 10 in
\cite{mkm}).

$(1)\Rightarrow(4)$ is immediate.

$(7)\Rightarrow(8)$ is immediate.

Similarly to $(3)\Leftrightarrow(10)$ (the implication
$(2)\Rightarrow(3)$ in Theorem 10 in \cite{mkm}) we have that
$(1)\Leftrightarrow(11)$, $(7)\Leftrightarrow(12)$ and
$(8)\Leftrightarrow(13)$.

$(4)\Rightarrow(2)$. Let $D$ be a dense subset of $C_k(X)$. By the
property $S_{fin}(\mathcal{D}^k,\mathcal{S}^p)$, for
 sequence $(D_i : D_i=D$ and $i\in \mathbb{N} )$
there is a sequence $(K_{i}: i\in\mathbb{N})$ such that for each
$i$, $K_{i}$ is finite, $K_{i}\subset D_{i}$, and
$\bigcup_{i\in\mathbb{N}}K_{i}$ is a countable sequentially dense
subset of $C_p(X)$.

$(2)\Rightarrow(9)$. Let $\mathcal{U}$ be an open $k$-cover of
$X$. Note that the set $\mathcal{D} :=\{ f\in C(X) :
f\upharpoonright (X\setminus U)\equiv 1$ for some
$U\in\mathcal{U}\}$ is dense in $C_k(X)$ and,  hence,
$\mathcal{D}$ contains a countable sequentially dense set $A$ in
$C_p(X)$ . Take $\{f_n: n\in \mathbb{N}\}\subset A$ such that
$f_n\mapsto \bf{0}$ ($n\mapsto \infty$) in $C_p(X)$. Let
$f_n\upharpoonright (X\setminus U_n)\equiv 1$ for some
$U_n\in\mathcal{U}$. Then $\{U_n: n\in \mathbb{N}\}$ is a
$\gamma$-subcover of $\mathcal{U}$, because of $f_n\mapsto
\bf{0}$. Hence, $X$ satisfies ${\mathcal{K}\choose\Gamma}$.

$(3)\Rightarrow(1)$. Let $(D_{i,j} : i,j\in \mathbb{N})$ be a
sequence of dense subsets of $C_k(X)$ and let $D=\{f_i : i\in
\mathbb{N}\}$ be a countable dense subset of $C_k(X)$.

% There exists $D'$ such that $[D']^p_{seq}=C_p(X)$
%and there exists $D''$ such that $\overline{D''}^k=C_k(X)$,
%then $D=D'\cup D''=\{f_i : i\in \mathbb{N}\}$ is a countable dense in $C_k(X)$
%and countable sequentially dense in $C_p(X)$.

For every $f_i\in D$ and $j\in \mathbb{N}$ consider
$\mathcal{U}_{i,j}=\{U_{h,i,j} :
U_{h,i,j}=(f_i-h)^{-1}(-\frac{1}{j},\frac{1}{j})\wedge
(U_{h,i,j}\neq \emptyset)$ for $h\in D_{i,j}\}$. Note that
$\mathcal{U}_{i,j}$ is an $k$-cover of $X$ for every $i,j\in
\mathbb{N}$. Since $X$ satisfies $S_{1}(\mathcal{K},\Gamma)$,
there is a sequence $(U_{h(i,j),i,j} : i,j\in \mathbb{N})$ such
that $U_{h(i,j),i,j}\in \mathcal{U}_{i,j}$, and
$\phi:=\{U_{h(i,j),i,j}: i,j\in\mathbb{N} \}$ is an element of
$\Gamma$.

We claim that $\{h(i,j) : i,j\in\mathbb{N}\}$ is a sequentially
dense subset of $C_p(X)$.

Fix $g\in C(X)$. There exists $(f_{i_k}: k\in \mathbb{N})$ such
that $f_{i_k}\rightarrow g$ $(k\rightarrow \infty)$ in $\tau_p$.
Then $(g-f_{i_k}) \rightarrow {\bf 0}$ in $\tau_p$. Show that $h
(i_k,j) \rightarrow g$ in $\tau_p$. Let $W=\langle g, A, \epsilon
\rangle$ be a base neighborhood of $g$ in $C_p(X)$, where $A$ is a
finite subset of $X$ and $\epsilon>0$. Since $\phi$ is a
$\gamma$-cover of $X$, then $\{U_{h(i_k,j),i_k,j} : k,j\in
\mathbb{N}\}$ is a $\gamma$-cover of $X$, too. There exists
$k',j'$ such that $\frac{1}{j'} < \frac{\epsilon}{2}$ and for
every $k>k', j>j'$ the following statements are true:
$(g-f_{i_k})(A) \subset (-\frac{\epsilon}{2};\frac{\epsilon}{2})$
and $(f_{i_k}-h(i_k,j))(A) \subset (-\frac{1}{j'};\frac{1}{j'})
\subset (-\frac{\epsilon}{2};\frac{\epsilon}{2})$. Notice, that
$((g-f_{i_k})+(f_{i_k}-h(i_k,j)))(A) = (g-h(i_k,j))(A) \subset
(-\epsilon;\epsilon)$. Then $h(i_k,j)\subset W$ for every $k>k',
j>j'$.

$(8)\Rightarrow(3)$. Let $\{\mathcal{U}_i: i\in
\mathbb{N}\}\subset \mathcal{K}$ and let $D=\{d_j: j\in
\mathbb{N}\}$ be a countable dense subset of $C_k(X)$. Consider
$D_i=\{f_{K,U,i,j}\in C(X):$ such that $f_{K,U,i,j}\upharpoonright
K\equiv d_j$, $f_{K,U,i,j}\upharpoonright (X\setminus U)\equiv 1$
where $K\in \mathbb{K}(X)$, $K\subset U\in\mathcal{U}_i\}$ for
every $i\in \mathbb{N}$. Since $D$ is a dense subset of $C_k(X)$,
then $D_i$ is a dense subset of $C_k(X)$ for every $i\in
\mathbb{N}$. By (8), there is a set $\{f_{K(i),U(i),i,j(i)} : i\in
\mathbb{N}\}$ such that $f_{K(i),U(i),i,j(i)}\in D_i$ and
$\{f_{K(i), U(i), i, j(i)} : i\in \mathbb{N}\}\in \Gamma^p_{\bf
0}$. Claim that a set $\{U(i) : i\in \mathbb{N}\}\in \Gamma$. Let
$K$ be a finite subset of $X$ and let
$W=[K,(-\frac{1}{2},\frac{1}{2})]$ be a base neighborhood of ${\bf
0}$. Since $\{f_{K(i), U(i), i, j(i)} : i\in \mathbb{N}\}\in
\Gamma^p_{\bf 0}$, there is $i'\in \mathbb{N}$ such that $f_{K(i),
U(i), i, j(i)}\in W$ for every $i>i'$. It follows that $K\subset
U(i)$ for every $i>i'$ and hence $\{U(i) : i\in \mathbb{N}\}\in
\Gamma$.

\end{proof}

We can summarize the relationships between considered notions in
next diagrams.

\begin{center}

$G_1(\mathcal{D}^k,\Gamma^p_{\bf 0}) \Leftrightarrow
G_{fin}(\mathcal{D}^k,\Gamma^p_{\bf 0}) \Rightarrow
G_1(\mathcal{D}^k,\Omega^p_0) \Rightarrow
G_{fin}(\mathcal{D}^k,\Omega^p_0)$ \\  \, \, $\Updownarrow$ \, \,
\, \, \,\, \,  \, \, \, $\Updownarrow$ \,\, \, \, \,\, \, \,
\, \, \, $\Updownarrow$ \,\, \, \, \,\, \, \,\, \, \, $\Updownarrow$ \\
$G_1(\Omega^k_0,\Gamma^p_{\bf 0}) \Leftrightarrow
G_{fin}(\Omega^k_0,\Gamma^p_{\bf 0}) \Rightarrow
G_1(\Omega^k_0,\Omega^p_0) \Rightarrow
G_{fin}(\Omega^k_0,\Omega^p_0)$ \\  \, \, $\Updownarrow$ \, \, \,
\, \,\, \,  \, \, \, $\Updownarrow$ \,\, \, \, \,\, \, \,
\, \, \, $\Updownarrow$ \,\, \, \, \,\, \, \,\, \, \, $\Updownarrow$ \\
$G_1(\mathcal{D}^k,\mathcal{S}^p) \Leftrightarrow
G_{fin}(\mathcal{D}^k,\mathcal{S}^p) \Rightarrow
G_1(\mathcal{D}^k,\mathcal{D}^p) \Rightarrow
G_{fin}(\mathcal{D}^k,\mathcal{D}^p)$ \\  \, \, $\Updownarrow$ \,
\, \, \, \,\, \,  \, \, \, $\Updownarrow$ \,\, \, \, \,\, \, \,
\, \, \, $\Updownarrow$ \,\, \, \, \,\, \, \,\, \, \, $\Updownarrow$ \\
$S_1(\mathcal{D}^k,\mathcal{S}^p) \Leftrightarrow
S_{fin}(\mathcal{D}^k,\mathcal{S}^p) \Rightarrow
S_1(\mathcal{D}^k,\mathcal{D}^p) \Rightarrow
S_{fin}(\mathcal{D}^k,\mathcal{D}^p)$ \\  \, \, $\Updownarrow$ \,
\, \, \, \,\, \,  \, \, \, $\Updownarrow$ \,\, \, \, \,\, \, \,
\, \, \, $\Updownarrow$ \,\, \, \, \,\, \, \,\, \, \, $\Updownarrow$ \\
$S_1(\mathcal{D}^k,\Gamma^p_{\bf 0}) \Leftrightarrow
S_{fin}(\mathcal{D}^k,\Gamma^p_{\bf 0}) \Rightarrow
S_1(\mathcal{D}^k,\Omega^p_0) \Rightarrow
S_{fin}(\mathcal{D}^k,\Omega^p_0)$ \\  \, \, $\Updownarrow$ \, \,
\, \, \,\, \,  \, \, \, $\Updownarrow$ \,\, \, \, \,\, \, \,
\, \, \, $\Updownarrow$ \,\, \, \, \,\, \, \,\, \, \, $\Updownarrow$ \\
$S_1(\Omega^k_0,\Gamma^p_{\bf 0}) \Leftrightarrow
S_{fin}(\Omega^k_0,\Gamma^p_{\bf 0}) \Rightarrow
S_1(\Omega^k_0,\Omega^p_0) \Rightarrow
S_{fin}(\Omega^k_0,\Omega^p_0)$

\medskip

Fig.~2. The Diagram of games and selectors of $(C(X), \tau_k,
\tau_p)$.

\end{center}

\medskip
\begin{center}
$G_1(\mathcal{K},\Gamma) \Leftrightarrow
G_{fin}(\mathcal{K},\Gamma) \Rightarrow G_1(\mathcal{K},\Omega)
\Rightarrow G_{fin}(\mathcal{K},\Omega)$ \\  \, \, $\Updownarrow$
\, \, \, \, \,\, \,  \, \, \, $\Updownarrow$ \,\, \, \, \,\, \, \,
\, \, \, $\Updownarrow$ \,\, \, \, \,\, \, \,\, \, \, $\Updownarrow$ \\
$S_1(\mathcal{K},\Gamma) \Leftrightarrow
S_{fin}(\mathcal{K},\Gamma) \Rightarrow S_1(\mathcal{K},\Omega)
\Rightarrow S_{fin}(\mathcal{K},\Omega)$

\medskip

Fig.~3. The Diagram of games and selection principles for a space
$X$ with $iw(X)=\aleph_0$ corresponding to selectors of $(C(X),
\tau_k, \tau_p)$.

\end{center}

\bigskip

\bibliographystyle{model1a-num-names}
\bibliography{<your-bib-database>}

%% Authors are advised to submit their bibtex database files. They are
%% requested to list a bibtex style file in the manuscript if they do
%% not want to use model1a-num-names.bst.

%% References without bibTeX database:
%%\bibliographystyle{plain}

% \begin{thebibliography}{00}

%% \bibitem must have the following form:
%%   \bibitem{key}...
%%

% \bibitem{}

% \end{thebibliography}

\end{document}